\documentclass[12pt,reqno]{amsart}

\usepackage{fancyhdr}
\usepackage{psfrag}
\usepackage{graphicx}
\usepackage{amsmath}
\usepackage{amsfonts}
\usepackage{amssymb}		
\usepackage{amsthm}
\usepackage{amscd}
\usepackage{indentfirst}
\usepackage{float}
\usepackage{latexsym}
\usepackage{graphics}
\usepackage{verbatim}

\renewcommand{\tilde}{\widetilde} 

\newcommand{\ft}{{\mathfrak t}}
\newcommand{\g}{{\mathfrak g}}

 \newcommand{\sF}{{\mathcal F}}

\newcommand{\A}{{\mathbb A}} 
 
\newcommand{\E}{{\mathbb E}} \newcommand{\F}{{\mathbb F}}
\newcommand{\G}{{\mathbb G}}

\newcommand{\Q}{{\mathbb Q}} \newcommand{\R}{{\mathbb R}}

\newcommand{\Z}{{\mathbb Z}} 

\newcommand{\bA}{{\bf A}}
\newcommand{\bB}{{\bf B}}
\newcommand{\bC}{{\bf C}}
\newcommand{\bD}{{\bf D}}
\newcommand{\bF}{{\bf F}}
\newcommand{\bE}{{\bf E}}
\newcommand{\bG}{{\bf G}}

\input{xy} \xyoption{all}

\newcommand{\ve}{\varepsilon}
\theoremstyle{plain}
\newtheorem{theorem}{Theorem}
\newtheorem{lemma}{Lemma}
\newtheorem{proposition}{Proposition}
\newtheorem{corollary}{Corollary}

\theoremstyle{definition}
\newtheorem{defn}{Definition}

\theoremstyle{remark}
\newtheorem*{remark}{Remark}
\newtheorem*{example}{Example}
\newtheorem*{ack}{Acknowledgement}

\begin{document} 

\title[Splitting fields of Zariski dense subgroups]
{On the splitting fields of generic elements\\ in Zariski dense
subgroups} 

\author[S. Pisolkar]{Supriya Pisolkar} \address{ Indian Institute of
Science, Education and Research (IISER),  Homi Bhabha Road, Pashan,
Pune - 411008, India} \email{supriya@iiserpune.ac.in} 

\author[C. S. Rajan]{C. S. Rajan} \address{School of mathematics, Tata
Institute of Fundamental Research, Homi Bhabha Road, Mumbai- 400005,
India} \email{rajan@math.tifr.res.in} 

\date{}

\begin{abstract}  Let $G$  be a connected, absolutely almost simple,
algebraic group defined over a finitely generated, infinite  field
$K$, and  let $\Gamma$ be a  Zariski dense subgroup of $G(K)$.  We
show, apart from some few exceptions,  that the commensurability class
of the field  ${\sF}$ given by the compositum  of  the splitting
fields of  characteristic polynomials of  generic elements of $\Gamma$
determines the group $G$ upto isogeny over the algebraic closure of
$K$. 
\end{abstract}

\maketitle

\section{Introduction}\label{Intro} The inverse spectral theory
problem in Riemannian geometry, is to recover properties of a
Riemannian manifold from the knowledge of the spectra of natural
differential operators associated to the manifold.  This problem has
been intensely studied in the context of Riemannian locally symmetric
spaces by various authors. 

In (\cite{PRwc})  G. Prasad and A. S. Rapinchuk introduced the notion
of weak commensurability for Zariski dense subgroups in absolutely
almost simple connected algebraic groups.
This notion is weaker than commensurability, but  they  showed
that weak  commensurability  of arithmetic lattices
implies commensurability in many instances. 
   
As an application, and assuming the validity of Schanuel's conjecture
on transcendental numbers in the case of  higher rank lattices, they
obtained commensurability results for  the corresponding locally
symmetric space defined by the arithmetic lattices, which are
isospectral with  respect to the Laplacian associated to  the
invariant metric acting on the space of smooth functions. 

In  (\cite{BPR}), the authors considered representation equivalent
lattices. If two uniform lattices are representation equivalent,  then
the corresponding Riemannian locally symmetric spaces are `strongly
isospectral'; in particular, they are isospectral for the Laplacian. 

One of the major advantages for considering the stronger but natural
notion of representation equivalence is that commensurability type
results for  representation equivalent lattices can be obtained
without appealing to Schanuel's conjecture. 

It can be seen by an application of the trace formula, that
representation equivalent uniform lattices are characteristic
equivalent.  For an algebraic group $G$ defined over a field $K$ and
an element $\gamma\in G(K)$, let  $P(\gamma, Ad)$ denote the
characteristic polynomial of  $\gamma$ in $G(K)$ with respect to the
adjoint representation $Ad$ of $G$ on its Lie algebra.  Let $G_1,
~G_2$ be connected absolutely almost simple algebraic groups defined
over a number field $K$. Two Zariski dense subgroups $\Gamma_i\subset
G_i(K), ~i=1,2$ are said to be {\em characteristic equivalent} if the
collection of characteristic polynomials $P(\gamma, Ad)$ of elements
$\gamma$ in $\Gamma_1$ (resp. $\Gamma_2$) are equal.

The  concept of characteristic equivalence is   stronger than that of
weak commensurability. From characteristic equivalence,  the
commensurability results of (\cite{PRwc}) follows more directly and
easily using the methods of (\cite{PRwc}).  

For an arithmetic lattice, the lengths of closed geodesics can be
expressed as a sum of logarithms of algebraic numbers.  In a
subsequent paper (\cite{PRl}),   Prasad and Rapinchuk studied the
compositum of  the fields generated over the field of algebraic
numbers,  by the lengths of closed geodesics on the locally Riemannian
symmetric space defined by the lattice.  Upon certain hypothesis on
the Weyl group and conditional on the validity of Schanuel's
conjecture on transcendental numbers, they showed that the fields are
quite different, provided the lattices are not commensurable.  

In this paper, we examine an analogous question, in the context of
characteristic equivalence. We consider the compositum of the
splitting fields of the characteristic polynomials of generic elements
of $G(K)$ contained in the  Zariski dense subgroup $\Gamma\subset
G(K)$. Apart from some few exceptions, we show that this field
determines the `commensurability class' of $\Gamma$, i.e., the group
$G$ upto isogeny over the algebraic closure of $K$. The advantage is
that the question in a purely algebraic setting, and Schanuel's
conjecture need not be invoked.

\section{Generic elements} Let $G$  be a connected, semisimple,
algebraic group defined over an infinite field $K$. Fix an algebraic
closure $\bar{K}$ of $K$. Let $T$ be a maximal torus in $G$ defined
over $K$, and let $X^*(T)={\rm Hom}_{\bar{K}}(T, \G_m)$ be the
character group of $T$ over $\bar{K}$.  The Galois group $G_K:={\rm
Gal}(\bar{K}/K)$ acts on $X^*(T)$ by the action, 
\[\sigma\chi(t)=\sigma(\chi(\sigma^{-1}(t))), \quad \sigma \in  G_K,
~t\in T(\bar{K}), ~\chi\in X^*(T).\]
Let $\g, ~\ft$ denote the Lie algebras of $G$ and $T$ respectively, 
 and denote by 
$\Phi = \Phi(G,T)\subset X^*(T)$ the root system
corresponding to the pair $(G, T)$:
\[ \g\otimes \bar{K}= \ft\otimes \bar{K}~~ \bigoplus_{\chi\in
  \Phi}\g_{\chi},\]
where $\g_{\chi}=\{X\in \g\otimes \bar{K}\mid Ad(t)(X)=\chi(t)(X),
~g\in G(\bar{K})\}$. 
The absolute Galois group
$G_K$ preserves  the root system $\Phi$, 
inducing  an injective  homomorphism 
\[ \theta_T : G_K\to \mbox{Aut}(\Phi(G,T)).\]
Suppose $g\in G(K)$ be a semi-simple regular element. 
Denote by   $T_g$ the connected component of its centralizer $Z(g)$ in
$G$.  Let $K_g$ denote minimal splitting field 
of ${T_g}$ in $\bar{K}$.
\begin{defn} Let $g\in G(K)$ be a regular semisimple element. 
Define $g$ to be  {\em generic} (or $K$-generic, or generic  $K$-regular), 
if the image $\theta_{T_g}(G_K)$ contains the Weyl group $W$.
\end{defn}  
Equivalently, the Galois group $G(K_g/K)$ contains a
subgroup  isomorphic to the Weyl group $W$. The notion of genericity
is sensitive to the underlying field $K$, and is not stable under base
change to a larger field $L$.

\begin{example} An element $g\in SL_n(K), ~n\geq 2$ is generic if and
  only if the Galois group over $K$ of the 
splitting field of its characteristic polynomial with
  respect to the natural representation is isomorphic to the symmetric
  group $S_n$ on $n$-symbols. 
\end{example}

\begin{remark}
Generic tori  was studied by Voskresenskii (\cite{V}) and 
called by him  as {\em tori without effect}. 
The nomenclature generic used here, 
stems from the fact that  generic (in the sense of algebraic
geometry) tori, satisfy this hypothesis.
We refer to \cite[Section 9]{PRg} and the references
contained in it for detailed discussions about generic tori and elements
and their properties. 
\end{remark}

\subsection{Splitting fields of characteristic polynomials}
We relate the notion of generic elements to characteristic
polynomials. 
Let $P(x, g, Ad)$ be the characteristic polynomial 
of the linear transformation $Ad(g)$, corresponding to the adjoint
representation $Ad$ of $G$ acting on its Lie algebra. For an element
$g\in G(K)$, denote by 
$K(g, Ad)$ the splitting field $P(x, g, Ad)$ over $K$.  

\begin{proposition}\label{Kgad} 
Let G be a connected, absolutely almost simple
algebraic group defined over an infinite field $K$. 
 Let  $g \in G(K)$ be a generic $K$-regular element
of infinite order.  With
the above notation, $K(g, Ad)=K_g$. 
\end{proposition}

\begin{proof}
Since $g$ is a regular semisimple element, $g\in
  T_g(K)=Z_G(g)^0(K)$. By definition, the tori $T_g$ splits over
  $K_g$. For $\alpha\in \Phi=\Phi(G,T_g)$, the value
  $\alpha(g)\in K_g$. Since the non-zero roots of the
  characteristic polynomial of the transformation $Ad(g)$ are given by
  $\alpha(g)$ for  $\alpha\in \Phi$, this implies that $K(g,
  Ad)\subset K_g$.

For the reverse inclusion, the elements $\alpha(g^n)\in K(g, Ad)$ for
$n\in \Z$ and $\alpha\in \Phi$.
Since $G$ is absolutely almost simple, the
action of the Weyl group $W$ on $X^*(T)\otimes \Q$ is
irreducible. From the correspondence between tori and its character
groups considered as modules for the Galois group, it follows that a
generic regular $K$-torus is irreducible over $K$, 
in that it cannot be written as an almost
direct product of two $K$-tori. Hence if $g\in G(K)$ is
generic $K$-regular, then  $T_g$ is a $K$-irreducible torus.

Since $g$ is of infinite order, the group generated by $g$ is Zariski
dense in $T_g$.  Hence every root $\alpha \in \Phi$ is $K(g,
Ad)$-rational character of the $K$-torus $T_g$. 
Since $\Phi$ is a basis for $X^*(T_g)\otimes \Q$, it
follows that every character of $T_g$ is $K(g, Ad)$- rational. Hence
$K_g\subset K(g, Ad) $. 

\end{proof}

\subsection{A theorem of Prasad-Rapinchuk} 
The following theorem of G. Prasad and A. Rapinchuk (\cite{PRr}, \cite{PRg}) 
guarantees the existence of generic regular elements in Zariski
dense subgroups: 
\begin{theorem}\cite[Theorem 9.6]{PRg}\label{prthm}
Let $G$ be a connected absolutely almost
simple algebraic group over a finitely generated, infinite field $K$. Let
$\Gamma \subset G(K)$ be a  Zariski dense subgroup in $G$. 
Then $\Gamma$ contains a $K$-generic element of
infinite order.
\end{theorem}

\section{Main theorems and a preliminary reduction}
Let $K$ be a finitely generated, infinite field, and 
$G$ be a connected, absolutely almost simple algebraic group
defined over  $K$. Fix an algebraic closure $\bar{K}$ of
$K$. By the classfication results of Killing, Cartan and Chevalley, 
the class of the group $G$ over  $\bar{K}$ upto isogeny, will be
called the type of $G$. 

Let $G_1, ~G_2$ be connected, 
 absolutely almost simple algebraic groups
defined over $K$. 
Call the pair $(G_1, G_2)$ to be {\em
  Weyl iso-trivial} if one of the following conditions hold: 

\begin{itemize}
\item Both $G_1$ and $G_2$ are of type $\bB_n$ or $\bC_n$ ($n\geq 2$). 

\item $n\geq 5$ is odd, and one of $G_1$ or $G_2$ is of type $\bB_n/\bC_n$  
and the other group is of type $\bD_n$. 

\item  One of  $G_1$ or $G_2$ are of  type $\bA_2$ and the other is of
  type
 $\bG_2$.
\end{itemize}

In these cases, we have the following
  isomorphisms: 
\begin{equation}\label{weylisotrivial}
\begin{split}
W(\bB_n)&\simeq W(\bC_n),\\
W(\bB_n)& \simeq W(\bD_n)\times \Z/2\Z, \quad n ~\mbox{odd}, ~n\geq 3,\\
W(\bG_2)& \simeq W(\bA_2)\times \Z/2\Z. 
\end{split}
\end{equation}
The first of the above equations follows from the fact that the root
systems of $\bB_n$ and $\bC_n$ are dual to each other. For a proof of
the second equation, see Corollary \ref{product}. The third equation
follows from the fact that the Weyl group of $\bG_2$ can be identified
with the dihedral group $D_6$.

\begin{defn} Two fields $L, ~M$ contained inside $\bar{K}$ are
  {\em commensurable} if both $L$ and $M$ are of finite degree over 
the intersection $L\cap M$. 
\end{defn} 

Equivalently, $L$ and $M$ are not commensurable if the compositum $LM$
inside $\bar{K}$ is of infinite degree over either $L$ or $M$. 

Our main theorem is the following: 
\begin{theorem}\label{maintheorem}
Let $G_1, ~G_2$ be connected,  
 absolutely almost simple algebraic groups
defined over a finitely generated infinite field $K$. Assume that 
they do not form a Weyl iso-trivial pair
as defined above. 

Suppose $\Gamma_i\subset G_i(K), ~i=1,2$ are
Zariski dense subgroups. For $i=1, 2$, 
let $\sF_i=\sF(\Gamma_i, K)$ be the subfield of
$\bar{K}$ given by the compositum of the fields $K_{\gamma}$ as
$\gamma \in \Gamma_i$ varies over the set of generic $K$-regular
elements in $\Gamma_i$. 

Suppose the fields  $\sF_1$ and  $\sF_2$ are commensurable. Then $G_1$
and $G_2$ are of the same Killing-Cartan type over $\bar{K}$. 
\end{theorem}

\subsection{Commensurable arithmetic lattices} In this subsection, we
relate Theorem \ref{maintheorem} to commensurability of arithmetic lattices. 
Let $K$ be a number field, and $\A_f$ be the ring of finite adeles of
$K$. Let $G$
be a semisimple algebraic group defined over
$K$, which for simplicity will be assumed to be of adjoint type. 
 
Let $F$ be a field with an embedding of $K$ into $F$. 
A Zariski dense subgroup $\Gamma\subset G(F)$ will be said to 
a $(G,K)$-arithmetic group, if it is commensurable with the image in
$G(F)$ of a  group of the form
$\Gamma_M:=G(K)\cap M$, where $M$ is a compact open subgroup of the
group $G(\A_f)$. 

A simple consequence of Theorem \ref{maintheorem} is the following
theorem: 
\begin{theorem}
Let $G_1, ~G_2$ be connected,  split, absolutely simple 
algebraic groups defined over a number field $K$. 
Assume that $G_1$ and $G_2$ do not form a Weyl iso-trivial pair. 

Suppose $\Gamma_i\subset G_i(K), ~i=1,2$ are
arithmetic subgroups. Assume that 
the fields  $\sF_1$ and  $\sF_2$ are commensurable. 

Then the lattices $\Gamma_1$ and $\Gamma_2$ are commensurable, i.e.,
there eixsts an isomorphism $\phi: G_1\to G_2$ defined over
$K$, such that $\phi(\Gamma_1)$ is commensurable with $\Gamma_2$. 
\end{theorem}

\begin{proof} By Theorem \ref{maintheorem}, 
the groups $G_1$ and $G_2$ are of the same Killing-Cartan type. 
Since they are  adjoint and split, there is an isomorphism $\phi:
G_1\to G_2$ defined over $K$. This preserves the group of finite adele
points and hence the lattices $\phi(\Gamma_1)$ and $\Gamma_2$ are
commensurable. 
\end{proof}

\begin{remark} It would be interesting to extend this theorem to 
not necessarily split groups and to arbitrary $S$-arithmetic
  lattices. 
\end{remark}

\subsection{Compositum of groups} 
The proof of Theorem \ref{maintheorem} is achieved by an argument
involving Galois theory and the structure of Weyl groups. 

\begin{defn}
A finite group $U$ is a {\em compositum} of the finite groups $U_1, \cdots,
U_r$ if there exists surjective maps $p_i: U\to U_i$ such that the
natural induced map   
$p: U\to \prod_{i=1}^rU_i$ is injective. 

 When all the $U_i$ are
isomorphic  to a given group say $U'$, we refer to $U$ as a compositum (or
composite) of $U'$. 
\end{defn}
The example we have in mind is given by the Galois group of a
compositum of Galois extensions: 

\begin{example}
Suppose $K$ is a field and  $L_1, \cdots, L_r$
are finite Galois extensions of $K$ contained inside 
an algebraic closure $\bar{K}$ of $K$. 
Let $U_i\simeq G(L_i/K)$ and $U\simeq G(L_1\cdots L_r/K)$. Then $U$ is
a compositum of the groups $U_i$.  
\end{example}

The following proposition says that the field theoretic example is not
restrictive: 
\begin{proposition}
Let $U$ be a compositum of the finite groups $U_1, \cdots,
U_r$, with respect to the maps $p_i: U\to U_i$. 
 Then  there exists a field $K$ and finite Galois 
extensions $L_1, \cdots, L_r$
of $K$ contained inside an algebraic closure $\bar{K}$ of $K$ 
such that $U_i\simeq G(L_i/K)$ and $U\simeq G(L_1\cdots L_r/K)$. 
\end{proposition}
\begin{proof} There exists a field $K_0$ and a Galois extension $L$ of
  $K_0$ with $G(L/K_0)\simeq U_1\times U_2\times \cdots \times
  U_r$. Let $K=L^U$, the field of $U$-invariants of $L$. For $i=1,
  \cdots, r$, let $R_i={\rm Ker}(p_i)$, and let $L_i=L^{R_i}$. Then
  $L_i$ is a Galois extension of $K$, with $G(L_i/K)\simeq U/R_i\simeq
  U_i$. 
\end{proof}
In this paper, we will mainly use the field theoretic language while dealing
with compositums of groups.

\subsection{A reduction}
The following (lack of) relationship between Weyl groups is the key
group theoretic statement required for the proof of Theorem
\ref{maintheorem}: 

\begin{theorem}\label{nonquotient}
Let $G_1, ~G_2$ be connected,  
 absolutely almost simple algebraic groups
defined over a finitely generated, infinite field $K$. 
Assume that $G_1$ and $G_2$ are
not isogenous over $\bar{K}$, and are not Weyl iso-trivial. 
Let $W_1$ (resp. $W_2$) be the Weyl group of $G_1$ (resp. $G_2$)
with respect to some maximal torus in $G_1$ (resp. $G_2$). 

Then, either $W_1$ (or
$W_2$) is not a    quotient of a compositum of  
normal subgroups of  $W_2$ (resp. $W_1$).
\end{theorem}

We now deduce Theorem \ref{maintheorem} from Theorem \ref{nonquotient}. 
\begin{proposition}\label{surjcrit}
Suppose $W_1$ is not a quotient of a compositum of normal subgroups of
$W_2$. Then $\sF_1\sF_2$ is of infinite degree over $\sF_2$.

In particular, Theorem \ref{maintheorem} follows from the validity 
of Theorem \ref{nonquotient}.
\end{proposition} 

\begin{proof} Suppose that $\sF_1\sF_2$ is a  finite extension of
$\sF_2$. Being algebraic extensions of $K$, this implies that there is a
finite Galois extension $L$ of $K$ such that $\sF_1\sF_2\subset L\sF_2$. In
particular, $\sF_1\subset L\sF_2$, and there is a 
surjection $G(L\sF_2/L)\to G(L\sF_1/L)$. 

We can assume that the groups $G_1$ and $G_2$ are split over $L$. 
Consider $\Gamma_1$ as a subgroup of $G_1(L)$. It is Zariski dense in
$G_1$ considered as an algebraic group over $L$. By Theorem
\ref{prthm}, there exists a generic $L$-regular element $\gamma\in
\Gamma_1$. Since $G_1$ is split over $L$,  
by (\cite[Lemma 4.1]{PRwc}), $G(L_{T_{\gamma}}/L)\simeq W_1$.  
The element $\gamma$ is generic $K$-regular, and the splitting field
of $T_{\gamma}$ is contained in $\sF_1$. 
 Hence there is a surjection $G(L\sF_1/L)\to W_1$, and
consequently from  $G(L\sF_2/L)\to W_1$. 

Without loss of generality, we can assume that $\sF_2$ is a
compositum of finitely many Galois extensions $K_{\eta}$ over $K$,
where $\eta\in \Gamma_2$ is a generic $K$-regular element in
$G_2(K)$. The Galois group $G(K_{\eta}/K)$ contains 
$W_2$ as a normal subgroup, and there
 is a surjection from  $G(L\sF_2/L)\to W_1$.

Since $G_2$ is split over $L$, by (\cite[Lemma 4.1]{PRwc}), the image
$\theta_{T_{\eta}}(G(LK_{\eta}/L))$ is contained in $W_2$. Since this
is a normal subgroup of $W_2$, we see that   $G(L\sF_2/L)$ is a compositum
of groups each of which is isomorphic to a normal subgroup of $W_2$. 

Theorem \ref{nonquotient} implies that 
$W_1$ cannot be a quotient of  $G(L\sF_2/L)$. 
 This yields a contradiction and proves the proposition. 
\end{proof}

\begin{remark} 
By Equation (\ref{weylisotrivial}), Theorem \ref{nonquotient}
does not hold when the groups $G_1$ and $G_2$ are iso-trivial. 
Our methods do not help in distinguishing between the fields $\sF_1$
and $\sF_2$ in these cases.
\end{remark}

In Section \ref{weylortho}, we study the finer structure of Weyl
groups of orthogonal groups, especially its normal subgroups, 
 and in the remaining sections we give a
proof of Theorem \ref{nonquotient}. The proof, in a rough sense,
rests on  `semisimple' properties of Weyl groups.

\section{Weyl groups of orthogonal groups} \label{weylortho}
In this section, we study the structure of the Weyl groups of 
orthogonal groups, especially the lattice  of its normal
subgroups.  

The Weyl groups of $\bB_n$ and $\bC_n$ are isomorphic. 
The Weyl group
$W(\bB_n)$ can be seen as signed permutations on the collection of 
basis vectors $\{e_1, \cdots, e_n\}$ of $\R^n$. The group
$V=(\Z/2\Z)^n$ sits as a normal subgroup of $W(\bB_n)$ by assigning to an
element $\ve=(\ve_1, \cdots, \ve_n)\in (\Z/2\Z)^n$, the
signed permutation
\begin{equation}\label{scaction}
 \sigma_{\ve}(e_i)=\ve_ie_i.
\end{equation}
There is an exact sequence, 
\begin{equation}\label{bseq}
 1\to  (\Z/2\Z)^n\to W(\bB_n)\to S_n\to 1.
\end{equation}
A splitting of this exact sequence is given by 
the symmetric group $S_n$ sits inside $W(\bB_n)$ as permutations
without changing the sign. This makes 
\[W(\bB_n)\simeq (\Z/2\Z)^n\rtimes S_n,\] 
a semidirect product of $S_n$ by $(\Z/2\Z)^n$. 
The conjugacy action
of $W$ on the abelian normal subgroup $(\Z/2\Z)^n$ descends to give
the standard permutation action of $S_n$ on  $(\Z/2\Z)^n$. Given two
elements $(\ve_1, \sigma_1)$ and $(\ve_2, \sigma_2)$ in the set
$V\times S_n$, the group multiplication is given as, 
\[ (\ve_1, \sigma_1)(\ve_2, \sigma_2)= (\ve_1+\sigma_1\ve_2,
\sigma_1\sigma_2),\]
where we have used the additive notation for the group multiplication
restricted to $V$. 
The inverse of $(\ve, \sigma)$ is $(-\sigma^{-1}\ve, \sigma^{-1})$.

The Weyl group of $\bD_n, ~(n\geq 4)$ is isomorphic to 
the subgroup of $W(\bB_n)$
consisting of permutations and sign changes involving only even number
of sign changes of the set $\{e_1, \cdots, e_n\}$, i.e., the number of
$\ve_i=-1$ in Equation (\ref{scaction}) is even. 

Consider the exact sequence of $S_n$-modules
\begin{equation}\label{dseq}
1\to V_e\to (\Z/2\Z)^n ~\xrightarrow{\phi}~ \Z/2\Z\to 1,
\end{equation}
where $ \phi(\ve)=\sum_{i=1}^n\ve_i$ and 
$V_e$ is the kernel of $\phi$. We have an isomorphism, 
\[ W(\bD_n)\simeq V_e\rtimes S_n.\]

\begin{lemma}\label{exportho}
Let  $n\geq 3$. If $n$ is a power of $2$, then the exponent of
$W(\bD_n)$ is equal to the exponent $e(S_n)$ of the symmetric group $S_n$.

In all other cases, the exponent of the Weyl groups
of $\bB_n$ and $\bD_n$ is twice the exponent of the symmetric group $S_n$.  
\end{lemma}
\begin{proof} Let $\sigma\in S_n$ be of order $k$. For any $\ve\in
 V= (\Z/2\Z)^n$, 
\[(\ve_{\sigma}, 1)=(\ve, \sigma)^k=(\ve+\sigma(\ve)+\cdots+\sigma^{k-1}(\ve),1).\]
Since every element in $V$ is of order $2$, the exponent of $W(\bB_n)$
or $W(\bD_n)$ is either $e(S_n)$ or $2e(S_n)$. 

Since the exponent is the least common multiple of the orders of the
elements in a group, in order for the exponent to be $2e(S_n)$, we need
to produce an element in the Weyl group of order $2^{l+1}$, where
$2^l$ is the maximum power of $2$ that divides $n$ (and
there is an element of that order in $S_n$). 

Suppose $n=2^lm$, with $m>1$. Consider the element $\sigma=(1\cdots
2^l)\in S_n$,  and $\ve=(1,0,\cdots,0,1)\in V_e\subset V$. In this
case, the element $\ve_{\sigma}$ will have co-ordinate $1$ at the
first $2^l$ indices, and all other co-ordinates are $0$. It
follows that $(\ve, \sigma)\in W(\bD_n)$ has order $2^{l+1}$. 

Let $n=2^l$. For the group $\bD_n$, if we
require an element of order $2^{l+1}$, without loss of generality, we can
take $\sigma$ and $\ve$ as in the foregoing paragraph. In this case, 
the element $(\sigma, \ve)$ has order $2^l$, and thus there are no
elements of order $2^{l+1}$ in $W(\bD_n)$.  

For the group $B_{2^l}$, consider $\sigma$ as above, and let
$\ve=(1,0,\cdots,0)$. Then  $\ve_{\sigma}=(1,1,\cdots,1)$, and the
element  $(\sigma, \ve)$ has order $2^{l+1}$. 
\end{proof}

Let $\Delta$ denote the $S_n$-stable subgroup $\Z/2\Z$ sitting diagonally 
inside $(\Z/2\Z)^n$. Let
$K_4$ denote the normal subgroup of
$S_4$, consisting of even permutations which are products of two
disjoint transpositions in $S_4$.
\begin{lemma}\label{invsub}
Let $n\geq 3$, and $H$ be a non-trivial normal subgroup of $S_n$. 
The only proper subspaces of
$(\Z/2\Z)^n$ which are invariant under $H$ 
are $\Delta$ and $V_e$.
\end{lemma}
Note that these subspaces are also invariant under $S_n$. 
\begin{proof}
We identify $V$ with $\F_2^n$, the $n$-dimensional
vector space over the field with two elements (denoted by $0$ and $1$)
$\F_2$. For any $\ve\in V$, let $Z_{\ve}$  denote the
subset of $\{1,\cdots, n\}$ consisting of the indices $i$, 
where $\ve_i=0$. Let $Z_{\ve}'$ denote the complement of $Z_{\ve}$ in
$\{1,\cdots, n\}$.

Let $W\subset V$ be a $H$-invariant subspace of $V$ which is not equal
to $\Delta$.
Suppose $\ve\in W$ is not an element of $\Delta$. Then both $Z_{\ve}$ and
$Z_{\ve}'$ are non-empty. 

We consider first the special case,  $n=4$, and
$|Z_{\ve}|=|Z_{\ve}'|=2$. Since the group $H\supset K_4$, and $K_4$ acts two
transitively on the set $\{1, \cdots,4\}$, it follows that $W=V_e$. 

We now consider the  general case, 
and  assume we are not in the above case when $n=4$.  Choose a subset
$I=\{i,j,k,l\}\subset \{1,\cdots, n\} $ 
of cardinality $4$ when $n\geq 4$, as follows: let $i\in Z_{\ve},
~j\in Z_{\ve}'$, and $k, l$ both are in the same set $Z_{\ve}$ or
$Z_{\ve}'$. 
Let $\sigma=(ij)(kl)$, and when $n=3$, let $\sigma$ be any even
permutation.  It can be seen that the element
$\sigma(\ve)\neq \ve$. 
 
Consider the non-zero element $\ve'=\ve+\sigma(\ve)\in W$. 
Since the map $\phi$ is $S_n$-invariant, $\ve'$ lies in the even
subspace $V_e$. 

Since $\sigma$ fixes the indices which are not in $I$, for any index
$m$ not in $I$, $\ve_m'=0$. Further,  $\ve_k'=\ve'_l=0$. 
Thus $\ve'$ is supported on two indices,
i.e., is of the form $e_{i_1}+e_{i_2}, ~i_1\neq i_2$, where $e_1,
\cdots, e_n$ form the standard basis for $\F_2^n$. 

Since $H$ acts doubly transitively on the set $\{1, \cdots,
n\}$, it follows that $W$ contains all elements of the form $e_i+e_j$
with $1\leq i< j\leq n$. Hence $W$ contains $V_e$, and this proves the
lemma. 

\end{proof}

We now describe the normal subgroups of the Weyl group $W(\bD_n)$. 
\begin{lemma}\label{normal}
Let $n\geq 3$. The proper normal subgroups of $W(\bD_n)$ are, 
\[ \Delta,~V_e, ~V_e\rtimes A_n, \quad \mbox{and} \quad V_e\rtimes
K_4 \quad\mbox{when}~ n=4.\]
\end{lemma}
\begin{proof} Let $ N$ be a proper, normal subgroup of $W=W(\bD_n)$.  
Denote by $\overline{N}$ the image of
$N$ in $S_n$ and $V_N=V\cap N$ the kernel of the projection map to
$S_n$. 

Suppose $\overline{N}=(e)$ is trivial. Then $N\subset V_e$ and has to be
an invariant subspace for the action of $S_n$. 
By Lemma \ref{invsub},  $V_N = \{0\},
\Delta$ or $V_e$. 

Assume now that $\overline{N}$ is non-trivial. 
We claim that
$V_N=V_e$. If this is so, then $N$ will be isomorphic to the
semi-direct product $V_e\rtimes \overline{N}$. 

For every $\sigma \in \overline{N}$, 
choose  $\ve_{\sigma} \in V_e$ such that $(\ve_{\sigma},\sigma) \in
N$. Since $N$ is normal, for any $\ve \in V_e$, the element 
\[(\ve,1)(\ve_{\sigma},\sigma)(\ve, 1)^{-1}( \ve_{\sigma}, \sigma)^{-1}=
(\ve+\ve_{\sigma}-\sigma(\ve), \sigma)(-\sigma^{-1}(\ve_{\sigma}),
\sigma^{-1})=(\ve-\sigma(\ve),1),\]
belongs to $N$, where $ 1$ denotes  the indentity
element in $S_n$. 

Thus  for any $\ve\in V_e$ and any $\sigma\in
\overline{N}$, the element $\ve-\sigma(\ve)\in V_N$.  Choose
$\ve=(1, 0,\cdots, 0)$.  When $n\geq 4$, 
let $i, j, k$ be distinct elements in the set
$\{2,\cdots, n\}$, and take $\sigma_i=(1,i)(j,k)$ belongs to
$\overline{N}$. 
When $n=3$, let $\sigma_i=(1,2,3)^i$.
The elements $\ve-\sigma_i(\ve)=(1,0,\cdots, 1_i,0,\cdots, 0)$ where
$1_i$ denotes $1$ at the $i$-th co-ordinate generate the subspace
$V_e$. 
Hence $V_N=V_e$ and this
proves the lemma. 
\end{proof}

\begin{corollary} \label{product}
Let $n=3$ or $n\geq 5$. Then the following holds: 
\begin{enumerate}
\item If $n$ is odd, $\Delta \cap V_e$ is the identity subgroup. Hence, 
\[ W(\bB_n)\simeq W(\bD_n)\times \Delta, \quad n ~\mbox{odd}.\]
\item When $n$ is even, $\Delta \subset V_e$, and $W(\bB_n)$ cannot be
  expressed as a product of two non-trivial groups. 

\end{enumerate}
\end{corollary}

\section{Proof of Theorem \ref{nonquotient}: Initial cases}
The proof of Theorem \ref{nonquotient} breaks up into different cases
depending on the relative structure of the Weyl group. In this
section, we use general group theoretic techniques to establish the
theorem in some cases. 

\subsection{Simple components}
Let $U$ be a finite group. A simple group $S$ is said to be a
component of $U$, if there exists subgroups $V_1\subset V_2\subset U$,
with $V_1$ normal in $V_2$ and $V_2/V_1\simeq S$. 
Denote by $JH_s(U)$ the simple components (counted
upto isomorphism and without multiplicity) 
that occur in $U$.
\begin{lemma}
Let $U$ be a compositum of groups $U_1,\cdots, U_r$.
 Then 
\[ JH_s(U)= \bigcup_{i=1,\cdots,r}JH_s(U_i).\]
\end{lemma}
\begin{proof} Since $U$ surjects onto each $U_i$ by definition of a
  compositum of groups,  we have 
$\bigcup_{i=1,\cdots,r}JH_s(U_i)\subset
  JH_s(U)$. 

Conversely suppose there exists subgroups $V_1\subset V_2\subset U$,
with $V_1$ normal in $V_2$ and $V_2/V_1\simeq S$. If for all
projections $U\to U_i$, the images of $V_1$ and $V_2$ coincide in
$U_i$, then $V_1=V_2$. Hence there exists some $i$, for which the
images of $V_1$ and $V_2$ are not equal. This implies that $S\in
JH_s(U_i)$ and this gives us the reverse inclusion. 
\end{proof}

\begin{corollary}\label{jh}
With the notation of Theorem \ref{nonquotient}, suppose 
that  $JH_s(W_1)$ is not a subset of  $JH_s(W_2)$.  
Then 
$\sF_1\sF_2$ is of infinite degree over $\sF_2$. 
\end{corollary}

\subsection{Simple components of Weyl groups}
From the explicit knowledge of Weyl groups of a simple root system,
the Weyl groups of the simple root systems  $\bA_{n-1}, ~\bB_n, ~\bD_n$
for $n\geq 5$ have the 
simple groups $\Z/2\Z$ and $A_n$ as the simple Jordan-Holder
components, 
where $A_n$ is the alternating group on $n$-symbols. We will
refer to these groups as JH-type
$\A_n$. 

Using the notation of the ATLAS, the Weyl groups of the 
exceptional groups $\bE_6$ (resp. $ \bE_7$, $\bE_8$)
 have $U_4(2)$ (resp. $ Sp_6(2), 
~\Omega_8^+(2)$) and $\Z/2\Z$ as the simple 
components. These will be called as groups of JH-type $\E$.  

The Weyl groups of the root systems $\bA_n ~(n=2,3), ~\bB_n/\bC_n ~(n=3,4),
~\bD_n~(n=4), ~\bF_4, ~\bG_2$ have $\Z/2\Z$ and $\Z/3\Z$ 
as the simple components. We will refer to this as
abelian of JH-type $\Z/3\Z$. 

Finally, $\Z/2\Z$ is the only simple component for the Weyl groups of
the root systems $\bA_1, ~\bB_2/\bC_2$. 
These will be referred to as of JH-type  $\Z/2\Z$. 

Hence we obtain the following corollary (we have assumed that $G_1$
and $G_2$ are not isogenous): 
\begin{corollary}\label{nacor}
With the above notation, suppose $G_1$ is of JH-type either
$\A_n$ or $\E$
and  $G_2$ is of JH-type either $\A_m$ with $n\neq
m$, $\E$ or $\Z/3\Z$. Assume that the groups $G_1$ and $G_2$ 
 are not isogenous over $\bar{K}$ (required when both are of $\E$-type).  Then   $\sF_1\sF_2$ is of
infinite degree over $\sF_1$ and $\sF_2$.
\end{corollary}

\begin{corollary}\label{a1cor}
With the above notation, suppose $G_1$ is not of JH-type $\Z/2\Z$, and 
$G_2$ is  of JH-type $\Z/2\Z$. Then   $\sF_1\sF_2$ is of
infinite degree over $\sF_2$.
\end{corollary}

\subsection{JH type $\Z/2\Z$}
\begin{corollary} \label{jh2case}
Suppose $G_1$ is of type $\bB_2$ or $\bC_2$, and $G_2$
  is of type $\bA_1$. Then $\sF_1\sF_2$ is of
infinite degree over $\sF_2$. 
\end{corollary}
\begin{proof} The Weyl group of $G_1$ is non-abelian whereas that of
  $G_2$ is abelian. Hence Theorem \ref{nonquotient} is valid in this
  case, and the corollary follows from Proposition \ref{surjcrit}. 
\end{proof}

\subsection{Exponent of Weyl groups}
For a finite group $H$, let $e(H)$ denote the exponent, i.e., the
least natural number  $n$ such that $h^n=e$ for all elements $h\in
H$. 
\begin{lemma}\label{exponent}
 Suppose that the exponent of $W_1$ does not divide the
  exponent of $W_2$. Then Theorem \ref{nonquotient} 
  holds. In particular, $\sF_1\sF_2$ is of
infinite degree over $\sF_2$. 
\end{lemma}
\begin{proof} Suppose 
  $U$ is a compositum of groups 
 $U_i ~i=1,\cdots, r$, each of which is
  isomorphic to a normal subgroup of  $W_2$. By definition of the
  compositum, 
\[ U \subset \prod_{i=1}^r U_i \subset  \prod_{i=1}^rW_2.\]
Thus the exponent of  $U$ divides the exponent of
$W_2$. By hypothesis, there cannot be a surjection from  $U$
to $W_1$. Hence the lemma follows. 
\end{proof}

This lemma allows us to take care of a few more cases, especially
when both $G_1$ and $G_2$ are of JH-type  $\Z/3\Z$. From the
calculation of the exponent of
the Weyl groups of orthogonal groups given by  Lemma
\ref{exportho}, and the calculation of the exponent of $W(\bF_4)$ by
Lemma \ref{expf4},  we have the following: 
\begin{lemma}
\begin{enumerate}

\item The exponents of $W(\bA_2)$ and $W(\bG_2)$ is $6$. 

\item The exponents of the Weyl groups of rank $3$-simple Lie algebras
 (of Killing-Cartan type  $\bA_3, \bB_3$ or $\bC_3$)  is $12$. 

\item The exponent of $W(\bD_4)$ is $12$. 

\item The exponent of $W(\bB_4)$ and $W(\bF_4)$ is $24$.
\end{enumerate}
\end{lemma}

Consequently we have the following corollary:
\begin{corollary}\label{expcor}
With the above notation,   $\sF_1\sF_2$ is of
infinite degree over $\sF_2$ in the following cases: 
\begin{enumerate}

\item The rank of $G_2$ is two,  and rank of $G_1$ is greater than
  two. 

\item The rank of $G_2$ is three, and  rank of $G_1$ is greater than
  three, except when $G_1$ is of type $\bD_4$. 

\item $G_2$ is of type $\bD_4$ and $G_1$ is either of type  $\bB_4$ and
  $\bF_4$.
\end{enumerate}
\end{corollary}

\begin{remark} It follows from Lemma \ref{exportho}, that Theorem
  \ref{maintheorem} holds when $G_1$ is of type $W(\bB_{2^l})$ and $G_2$
  is of type $W(\bD_{2^l})$, $l \geq 2$.
\end{remark} 

 \subsection{JH-type $\Z/3\Z$}
We now prove the theorem when $G_1$ is of type $\bD_4$ and $G_2$ is of
rank $3$. The group $G_2$ is of type either $\bA_3, \bB_3$ or $\bC_3$.
The Weyl group $W_2$ of $G_2$ is isomorphic to either $S_4$ or
$S_4\times \Z/2\Z~(\simeq W(B_3))$.  

For a group $H$, let $H^{(r)}$ denote the $r$-th derived subgroup of $H$. 
We have, 
\[ S_4^{(1)}\simeq A_4, \quad   S_4^{(2)}\simeq K_4 \quad\mbox{and}
\quad S_4^{(3)}\simeq (1).\]
Suppose there is a surjection $G(L/K)\to W_1$. For each $t\geq 1$, the
$r$-th derived group $G(L/K)^{(r)}$ surjects onto $W_1^{(r)}$. 

The group $G(L/K)$
embeds into $\prod_{i=1}^rG(L_i/K)$, where each 
$G(L_i/K)$ is a normal subgroup of $W_2$.
Since $W_2^{(3)}$ is the trivial group, it follows that
$G(L/K)^{(3)}$ is trivial. 

The group $W(\bD_4)$ surjects onto $S_4$. Hence there is a surjection
$W(\bD_4)^{(2)}\to K_4$. 
Since $W(\bD_4)^{(2)}$ is a normal subgroup of  $W(\bD_4)$, 
by Lemma \ref{normal},
$W(\bD_4)^{(2)}$ is non-abelian. Hence  $W(\bD_4)^{(3)}$ is non-trivial,
and this proves Theorem \ref{nonquotient} in this case. 

\subsection{}
Corollaries \ref{nacor}, \ref{a1cor}, \ref{jh2case} 
 and \ref{expcor} prove Theorem
\ref{maintheorem} in many cases. The cases left in order to prove
Theorem \ref{maintheorem} are the following: 
\begin{enumerate}
\item  Both $G_1$ and $G_2$ are of JH-type   $\A_n$.

\item $G_1$ and $G_2$ are of type  $\bB_4$ or
  $\bF_4$.
\end{enumerate}

\section{A vs B and D}
We consider the case where $G_1$ is either of type $\bB_n$ or $\bD_n$ and
$G_2$ is of type $\bA_{n-1}$.

\begin{proposition}
In the notation of Theorem \ref{maintheorem}, 
let $G_1$ be of type  either $\bB_n/\bC_n$ or $\bD_n$ for $n\geq 3, n\neq
4$. Let $G_2$ be of type $\bA_{n-1}$. Then $\sF_1\sF_2$ is of
infinite degree over $\sF_2$.
\end{proposition}

\begin{proof}
By Proposition \ref{surjcrit}, it is enough to show that
there is no surjective map from $N$ to  $W_1=W(G_1)$, where 
$N$ is the Galois group of a 
compositum of Galois extensions  $L_i/K$ for $i=1,\cdots, r$, and each
Galois group $G(L_i/K)$ is isomorphic to a normal subgroup of $S_n$. 

Upto reordering of the
indices, assume that  $G(L_i/K)\simeq S_n$ for $i=1, \cdots, s$
and $G(L_i/K)\simeq A_n$ for $i>s$. For $i=1, \cdots, s$,
let $E_i$ be the unique quadratic extension of $K$ contained inside
$L_i$, and $E$ be the compositum of the $E_i$. There is an exact
sequence, 
\[ 1\to G(L/E)\to G(L/K)\to G(E/K)\to 1, \]
where $G(E/K)\simeq (\Z/2\Z)^t$ for some $t\geq 0$. 
The group $G(L/E)$ is a compositum of the Galois groups $G(L_iE/E)$
each of which is isomorphic to the group $A_n$. Since
$A_n$ is simple, there are no proper Galois extensions of $E$
contained inside $L_iE$. Hence,  
\[ G(L/E)\simeq A_n^u\]
for some natural number $u\leq r$.

Suppose there is a surjection $\phi$ from $N$ to $W_1$. By Lemma
\ref{normal}, $W_1$ has no
normal subgroups isomorphic to $A_n$, the subgroup $G(L/E)$ of
$N$ lies in the kernel of $\phi$. This implies that the map $\phi$
factors via the abelian group $G(E/K)$, which is impossible as $W_1$
is non-abelian. 

Hence there does not exist any surjective homomorphism from $N$ to
$W_1$. By Proposition \ref{surjcrit}, the proposition is proved. 
\end{proof} 

\section{D vs B}
We consider the case where $G_1$ is either type $\bB_n,$ or $\bC_n$,  and
$G_2$ is of type $\bD_n$, and $n$ is even. 

Let $N$ be the Galois group of a 
compositum $L$ of Galois extensions  $L_i/K$ for $i=1,\cdots, r$, with
each $G(L_i/K)$ a normal subgroup of $W_2=W(\bD_n)$. 
We need to show that there is no surjective
map from $N=G(L/K)$ to $W_1=W(\bB_n)$. 

The Galois group $G(L_i/K)$ is a  normal subgroup of $W_2$. There is
an extension, 
\begin{equation}\label{extension}
 1\to G(L_i/E_i)\to G(L_i/K)\to G(E_i/K)\to 1,
\end{equation}
where $G(L_i/E_i)$ is the maximal abelian normal subgroup of
$G(L_i/K)$ and $G(E_i/K)$ is isomorphic to either the trivial, or the
alternating group $A_n$ or the symmetric group $S_n$ on
$n$-symbols. Let $E$ be the compositum of the fields $E_i$. 

Over $E$, the field $L$ can be considered as the compositum of the abelian
extensions $L_iE$ of $E$. Hence $G(L/E)$ is abelian and there is an
exact sequence, 
\begin{equation}\label{extension1}
 1\to G(L/E)\to G(L/K)\to G(E/K)\to 1.
\end{equation}
This induces an action of $G(E/K)$ on $G(L/E)$. 
Since $E/K$ is a compositum of Galois extensions of $K$, 
there is a natural inclusion of the Galois groups, 
\begin{equation}\label{comptoprod1}
G(E/K) \subset \prod_{i=1}^rG(E_i/K), 
\end{equation}\label{comptoprod2}
Similarly, the Galois group $G(L/E)$ is contained in a product of groups, 
\begin{equation}
G(L/E)\subset \prod_{i=1}^rG(L_iE/E)\simeq  \prod_{i=1}^rG(L_i/L_i\cap
E)\subset \prod_{i=1}^rG(L_i/E_i), 
\end{equation}
where the equality arises from restriction. 

\begin{lemma}\label{compaction}
The  action of $G(E/K)$ on $G(L/E)$ defined by the extension Equation
(\ref{extension1}) is compatible via the inclusions defined by
Equations (\ref{comptoprod1}) and (\ref{comptoprod2}), 
with the component wise action of 
$ \prod_{i=1}^rG(E_i/K)$ on  $\prod_{i=1}^rG(L_i/E_i)$.
\end{lemma}

\begin{proof} Let $\sigma\in G(E/K)$ and $\tau\in G(L/E)$. The action
  of $\sigma$ on $\tau$ is given by
  $\sigma(\tau)=\tilde{\sigma}\tau\tilde{\sigma}^{-1}$, where
  $\tilde{\sigma}$ is any extension of $\sigma$ to an automorpism of
  $L/K$. Then 
\[ \sigma(\tau)_i=\tilde{\sigma}_i~\tau_i~\tilde{\sigma}_i^{-1},\]
where the subscript index $i$ denotes  the restriction of the Galois
automorphisms to $L_iE$. Considering $G(L_iE/E)\subset G(L_i/E_i)$,
this action depends only on the restriction of $\sigma$ to
$E_i/K$, and this proves the lemma. 
\end{proof}

\subsection{Image of abelian part} In order to prove Theorem
\ref{nonquotient}, we first show that the image of the `abelian part'
$G(L/E)$ of $G(L/K)$ does not cover the abelian part $V\subset W_1$:
\begin{lemma}\label{abelianimage}
Suppose $\theta: G(L/K)\to W(\bB_n)$ is a surjective homomomorphism. Then 
$\theta(G(L/E))\subset V_e$, where $V_e$ is the `even' subspace of
$V=(\Z/2\Z)^n$ defined by Equations (\ref{bseq}) and (\ref{dseq}). 
\end{lemma}
\begin{proof}
Since $G(L/E)$ is an abelian normal subgroup of $G(L/K)$, the  
image $\theta(G(L/E))$ is contained inside $V$, the maximal abelian
normal subgroup of $W(\bB_n)$. Let $\bar{\theta}$ denote the induced,
surjective map
from $G(E/K)\to S_n$. 

Let the indexing be such that for $i=1,\cdots, s$, the extension
$E_i/K$ is non-trivial,  and for $s<i\leq r$ the extension $E_i/K$ is
trivial. This is the same as saying that for $i\leq s$ (resp. $i>s$), 
the extension $L_i/K$ is non-abelian (resp. nonabelian).  The $G(E/K)$-module
$G(L_iE/E)$ can be considered as a submodule of $G(L_i/E_i)\subset V_e$ 
with respect to the projection $G(E/K)\to G(E_i/K)$.  

For $i\leq s$, by Lemma \ref{normal} on  the
classification of normal subgroups of $W(\bD_n)$,  $G(L_i/E_i)\simeq V_e$, this 
as $G(E_i/K)$-module. We continue to denote by $V_e$ the isomorphism
class of $G(L_i/E_i)$ as $G(E/K)$-module.  

For $i\geq s$, since
$G(E_i/K)$ is trivial, the group $G(E/K)$ acts trivially on
$G(L_iE/E)$. 
In particular, $G(L_iE/E)$ is a direct sum of $\Z/2\Z$-groups with
trivial $G(E/K)$-action. Thus, 
\[ \prod_{i=1}^rG(L_iE/E)\simeq V_e^s\times
\prod_{i>s}^rG(L_iE/E)\simeq  V_e^s\times (\Z/2\Z)^t,\]
as $G(E/K)$-modules for some $t\geq 0$.
 
For $i\geq s$, the quotient
module $V_e/\Delta$ is a simple $G(E/K)$-module, since it is simple as
$A_n$ or $S_n$-module. Hence the quotient $G(E/K)$-module   
\[ \left(\prod_{i=1}^rG(L_iE/E)\right)/\Delta^s\simeq  (V_e/\Delta)^s\times
(\Z/2\Z)^t,\]
is a semisimple $G(E/K)$-module. 

Consider the $G(E/K)$-module $G(L/E)\cap \Delta^s$. Since this is
trivial as a $G(E/K)$-module, and $\bar{\theta}$ is a surjection onto
$S_n$,  the image $\theta(G(L/E)\cap
\Delta^s)\subset V$ is trivial as a $S_n$-module. It follows that  
\[\theta(G(L/E)\cap \Delta^s)\subset \Delta.\]
Now the  $G(E/K)$-module $G(L/E)/(G(L/E)\cap \Delta^s)$ is a submodule
of the semisimple $G(E/K)$-module $(V_e/\Delta)^s\times
(\Z/2\Z)^t$. This implies that it is semisimple, and there are numbers
$s', ~t'\geq 0$ such that  
\[G(L/E)/(G(L/E)\cap \Delta^s)\simeq  (V_e/\Delta)^{s'}\times
(\Z/2\Z)^{t'},\]
as a $G(E/K)$-module. Each of the $G(E/K)$-summands on the right is of
cardinality at most $2^{n-2}$. 

Again by the surjectivity $\bar{\theta}: G(E/K)\to S_n$, 
the image  
$$\theta(G(L/E)/(G(L/E)\cap \Delta^s))\subset V/\Delta,$$
 is a sum of $S_n$-modules, each of which has cardinality at
most  $2^{n-2}$. By Lemma \ref{invsub}, the image of each irreducible
summand is contained inside $V_e/\Delta$. Hence the image 
 $\theta(G(L/E))$ is contained inside $V_e$, and this proves the
 lemma. 
\end{proof}

\subsection{A splitting property} Although each of the groups
$G(L_i/K)$ is a semidirect product of $G(E_i/K)$ by $G(L_i/E_i)$, it
is not clear that $G(L/K)$  is a semidirect product of $G(E/K)$ by
$G(L/E)$. The following lemma serves as a replacement for this
property. 
\begin{lemma} \label{sd}
Given the extension, 
\[
 1\to G(L/E)\to G(L/K)\to G(E/K)\to 1,\]
there is a group $Q$ satisfying the following: 

\begin{enumerate}
\item $Q$ is obtained from $G(E/K)$ recursively by a sequence of
  central extensions with kernel $\Z/2\Z$, i.e., there is a sequence
  of maps 
$$Q=Q_k\to Q_{k-1}\to\cdots \to Q_0=G(E/K)$$
such that  for $i=0, \cdots, k-1$, there is a central extension
\[ 0\to \Z/2\Z \rightarrow Q_{i+1} \rightarrow Q_i\to 1.\]

\item $Q$ splits the extension given by Equation (\ref{extension1})
\[
 1\to G(L/E)\rightarrow G(L/K)\rightarrow G(E/K)\to 1, \]
i.e., there is a map $s: Q\to G(L/K)$ such that composed with the
projection to $G(E/K)$ is the map $\pi:Q\to Q_0=G(E/K)$. 
\end{enumerate}
\end{lemma}

\begin{proof} 
We prove this by induction on $r$. From the structure of
  Weyl groups, the statement is true when $r=1$. 
Denote by $L'$
  (resp. $E'$) the compositum of the fields $L_1, \cdots, L_{r-1}$
  (resp. $E_1, \cdots, E_{r-1}$). By induction, we assume the
  existence of a extension $\pi': Q' \to G(E'/K)$ satisfying the
  properties of the lemma for the extension $G(L'/K)\to G(E'/K)$. 
We divide the proof into different
  cases, and first consider the cases where the extensions $L_i/K$ are
  non-abelian, i.e., $E_i\neq K$.    

\noindent{(i) $E'\neq E$.} Suppose $E'\cap E_r=F$, and $E_r$ is a
non-abelian extension of $K$.  From the structure
of normal subgroups of $W(\bD_n)$, $G(E_r/K)$ is isomorphic to either
$A_n$ or $S_n$. It follows that $E'\cap E_r=F$ is
either $K$ or  a quadratic extension of $K$. 

We claim that $L'\cap L_r=F$. Suppose not. Since $L'\cap L_r$ is a
normal extension contained inside $L_r$, if it is not equal to $F$,
then it has to contain $E_r$. This is equivalent to a surjective
homomorphism from $G(L'/K)\to G(E_r/K)$. Since $G(E_r/K)$ does not
contain any abelian normal subgroups, the map factors 
$G(L'/K)\to G(E'/K)\to G(E_r/K)$.  But this means $E'\supset E_r$,
contradicting the assumption that  $E'\cap E_r=F$. This proves the
claim.

By definition, $G(E/K)$ is the subgroup of $G(E'/K)\times G(E_r/K)$
consisting of pairs $(\sigma,\tau)\in G(E'/K)\times G(E_r/K)$
satisfying $\sigma|_{E'\cap E_r}=\tau|_{E'\cap E_r}$. Let $Q$ be the
extension of $G(E/K)$ obtained as a pullback  restricted to $G(E/K)$ 
of the extension
$Q'\times G(E_r/K)\to G(E'/K)\times G(E_r/K)$.
The composite map
\[ Q \to Q'\times G(E_r/K)\to G(L'/K)\times G(L_r/K), \]
splits the extension $ G(L'/K)\times G(L_r/K) \to
G(E'/K)\times  G(E_r/K)$, by the inductive hypothesis.  

To check that the image of $Q$  lands inside the subgroup $G(L/K)$, we
need to show that  for $(\sigma,\tau)\in G(L'/K)\times G(L_r/K)$, 
then $\sigma|_{L'\cap L_r}=\tau|_{L'\cap L_r}$. But $L'\cap L_r=E'\cap
E_r$.  Since $Q$ is an
extension of $G(E/K)$, the projection of the image of $Q$ to the group
$G(E'/K)\times  G(E_r/K)$ lands inside $G(E/K)$. 
This proves the compatibility of $\sigma$ and
$\tau$, and establishes the lemma in this case. 

\noindent{(ii) $E'=E$.} There is an isomorphism
$G(L/L')\simeq G(L_rE/L'\cap L_rE)$, and the latter group is a
module for the $G(E_r/K)$, contained inside $V_e$. By Lemma
\ref{invsub}, $G(L_rE/L'\cap L_rE)$ is isomorphic to either $V_e$ or
$\Delta$ or $0$. 

If it is $0$, then $L=L'$, and there is nothing to
prove. 

We argue now using cohomological language. Let $c$ denote the
cohomology class  in 
$H^2(G(E/K), G(L/E))$ corresponding to the extension 
given by Equation (\ref{extension1}).
We are required to show that there exists a  extension 
$p: Q\twoheadrightarrow G(E/K)$ of
$G(E/K)$ satisfying Property (1),  such
that the pullback $\pi^*(c)\in H^2(Q, G(L/E))$ is trivial, i.e., the
pullback to $Q$ of the extension given by Equation (\ref{extension1})
splits. 

Suppose  $G(L_rE/L'\cap L_rE)\simeq V_e$ as $G(E_r/K)$-module. 
In this case, the fields $L'$ and $L_rE$ are linearly disjoint over
$E$, and there is an isomorphism of $G(E/K)$-modules, 
$ G(L/E)\simeq G(L'/E)\times G(L_rE/E)$.
Hence,
 \[H^2(G(E/K),  G(L/E))\simeq H^2(G(E/K),  G(L'/E))\oplus 
H^2(G(E/K),  G(L_rE/E)).\]
By Lemma \ref{compaction}, the 
cohomology class in  $H^2(G(E/K), G(L/L'))$ defined by the extension
is inflated from the
cohomology class in  $H^2(G(E_r/K), G(L_r/L'\cap L_r))$. 
By Lemma \ref{normal}, the extension of $G(E_r/K)$ defined
by $G(L_rE/L'\cap L_rE)$ splits. The lemma follows by taking $Q=Q'$ and
$\pi=\pi'$.

We come to the interesting case, when 
\[ G(L/L')\simeq G(L_rE/L_rE\cap L')\simeq \Delta.\]
As a group $\Delta\simeq \Z/2\Z$. Since $G(L/L')$ is a
$G(E/K)$-module, by projecting to each  $G(L_iE/E)$, it is seen that $G(E/K)$
acts trivially on $G(L/L')\simeq  (\Z/2\Z)^t$ for some $t\geq 0$. 
The exact sequence, 
\[1\to G(L/L')\to G(L/E)\to G(L'/E)\to 1,\]
considered as $G(E/K)$-modules, 
yields the following exact sequence of cohomology groups, 
\begin{equation}\label{longexact}
 H^2(G(E/K), G(L/L'))\to H^2(G(E/K), G(L/E))\to H^2(G(E/K), G(L'/E)).
\end{equation}
By induction, there exists a  extension $\pi': Q'\to G(E/K)$ satisfying
Property (1) of the lemma,  such
that the class $\pi'^*(c')\in  H^2(Q', G(L'/E))$ is trivial. 
Pulling back the sequence
given by Equation (\ref{longexact}) to $Q'$, it follows that 
the cohomology class $\pi'^*(c')\in  H^2(Q',
G(L/E))$ is the image of a cohomology class $c''$ in  $H^2(Q',
G(L/L'))$. 

Since $G(E/K)$ acts trivially on $G(L/L')$, the cohomology
class of $c''$ defines a central extension, 
\[ 1\to G(L/L') \to Q\to Q'\to 1.\]
The cohomology class $c''$ 
can be killed by going to the central extension
$Q$. Since $G(L/L')\simeq (\Z/2\Z)^t$ for some $t\geq 0$, 
this proves both parts of the lemma.  

This proves the lemma, except in the cases when the field extensions
$L_i/K$ are abelian.  In this case, the Galois group $G(L_i/K)$ is
isomorphic to either $\Z/2\Z$ or   $(\Z/2\Z)^{n-1}$. Let $L'$
(resp. $L''$) be the
compositum of all the non-abelian (resp. abelian) extensions $L_i/K$.  
The field $L''$ can
be decomposed as $L''=F'F''$, where $F'=F\cap L'$ and $F''$ is
disjoint from $L'$. The Galois group $G(L/K)\simeq G(L'/K)\times
G(F''/K)$, and this maps to $G(E'/K)\simeq G(E/K)$, where $E$ and $E'$
are as notation used in the foregoing paras. Taking $Q=Q'$, satisfies
the conditions of the lemma, and this proves the lemma in all cases.   
\end{proof}

\subsection{Proof of Theorem \ref{nonquotient}}
We now come to the proof of Theorem \ref{nonquotient}, in the case when 
$G_1$ is either type $\bB_n, ~\bC_n$  and
$G_2$ is of type $\bD_n$, and $n$ is even. Suppose there is a surjective
map  $\theta: G(L/K) \to W(\bB_n)$. By Lemma \ref{abelianimage}, the
image $\theta(G(L/E))\subset V_e$. 

Consider the homomorphism $\theta\circ s: Q\to W(\bB_n)$, where $Q\to
G(L/K)$ is as in Lemma \ref{sd}. There is a sequence
  of maps 
$$Q=Q_k\to Q_{k-1}\to\cdots \to Q_0=G(E/K),$$ where each of the
extensions $Q_i\to Q_{i-1}$ is a central $\Z/2\Z$-extension. Arguing 
as in the proof of Lemma \ref{abelianimage}, by downward induction on
$i$,  
 it follows that the image under $\theta\circ s $
of the kernel of the projection $p: Q\to G(E/K)$ lands inside the
subgroup $\Delta\subset W(\bB_n)$. 

Going modulo $\Delta$, we have a map $G(E/K)\to W(\bB_n)/\Delta$, which
we continue to denote by $\theta$. Since $V/\Delta$ has no
$S_n$-invariant subspace, arguing as in the proof of Lemma
\ref{abelianimage}, it follows that the image of $G(E/K)$ 
intersects $V/\Delta$ trivially. It follows that 
 the image of $Q$ intersection  $V$ is contained inside $\Delta$. 

Going modulo $V_e$, we have that the map $\theta : G(L/K)\to W(\bB_n)$ 
factors via $Q\to (V/V_e)\times S_n$. But the projection of the image
of $Q$ to the first factor $V/V_e$ is trivial. Hence it follows that
$\theta$ cannot be a surjection, and Theorem \ref{nonquotient} follows
in this case.

\section{$\bB_4$ vs $\bF_4$}
\subsection{$\bF_4$-root system}
We study now the root system $\bF_4$ and its Weyl group $W(\bF_4)$
following (\cite{Bou}). 
Consider the four dimensional vector space over $\R^4$ with standard
basis vectors $e_1, \cdots, e_4$. Let $L_0=\Z^4$, and $L_1$ denote the
sublattice of $L_0$ consisting of vectors $x\in Z^4$ such that
$||x||^2\in 2\Z$. Denote by
$L_2=L_0+\Z\left(\frac{1}{2}(\sum_{i=1}^4e_i\right)$, the dual lattice of
  $L_1$ inside $\R^4$. 

The root system $R$ of $\bF_4$ can be described as those elements $x\in
L_2$ with $||x||^2$ equal to either $1$ or $2$.  The element $\ve$
(with a subscript) will denote either $1$ or $-1$. 
The root system of $\bF_4$ has 24 long roots of the form $\varepsilon_ie_i+
\varepsilon_je_j, ~1\leq i <j\leq 4$.  For the short roots, $8$ of them
are of the form $\varepsilon_ie_i, ~1\leq i\leq 4$,
 and $16$ of them are of the form $(\sum_{i=1}^4\varepsilon_ie_i)/2$. 

The long (resp. short) roots form a root system $R_l$ (resp. $R_s$) 
of type $\bD_4$. There is a homomorphism $W(\bF_4)\to {\rm
  Aut}(R_l)\simeq  {\rm Aut}(\bD_4)$,
which is injective since $R_l$ is of the same rank as $\bF_4$. Since
$R_l$ generates $L_1$, any automorphism of $R_l$ will leave $L_1$ and
$L_2$ invariant. Hence,   
\[ W(\bF_4)\simeq {\rm Aut}(\bD_4)\simeq W(\bD_4)\rtimes S_3, \]
where the second isomorphism follows from the fact that 
the outer automorphism group of $\bD_4$ is
isomorphic to $S_3$. It follows that the order of $W(\bF_4)$ is $1152$. 

The Weyl group $W(\bD_4)_l$ (resp. $W(\bD_4)_s$)   generated by the
reflections in the long (resp. short) roots are normal subgroups in
$W(\bF_4)$. The Weyl group $U_l\simeq S_3$ 
of the $\bA_2$ root system formed by the two
long vectors $e_2-e_1, e_3-e_2$ surjects onto the outer automorphism
group of $(\bD_4)_s$ with base given by $-e_1, -e_2,-e_3,
(e_1+\cdots+e_4)/2$. Similarly, the Weyl group 
$U_s\simeq S_3$ of the $\bA_2$-root
system formed by the two short vectors $u=(e_1+e_2+e_3+e_4)/2, -e_4$ 
surjects onto the outer automorphism
group of $(\bD_4)_l$ with base given by $-e_1-e_3, -e_2+e_3, e_2-e_4,
e_2+e_4$. Since the base roots defining $U_l$ and $U_s$ are mutually
orthogonal, the groups $U_l$ and $U_s$ commute with each other. This
yields a semidirect product, 
\begin{equation}\label{f4sd}
 W(\bF_4)\simeq A\rtimes (S_3\times S_3).
\end{equation}
The group $A$ is of cardinality $32$
and is the intersection of the two normal subgroups $W(\bD_4)_l$ and
$W(\bD_4)_s$.  
 In
$W(\bD_4)_l$, the normal subgroup $A$ is given by the semidirect product
$V_e\rtimes K_4$, where $K_4$ is the normal subgroup of order $4$
contained inside $S_4$, and $V_e$ is the even
subspace contained inside $(\Z/2\Z)^4$. 
\begin{lemma}\label{expf4}
The exponent of $W(\bF_4)$ is $24$.
\end{lemma}
\begin{proof} The exponent of the group $A$ is $4$, and thus the
  exponent of  $W(\bF_4)$ can be atmost 24. On the other hand, the
  roots $e_1-e_2, e_2-e_3, e_3-e_4, e_4$ forms a base for a
  $\bB_4$-root system. By Lemma \ref{exportho}, 
the exponent of the $W(\bB_4)$ is 24 and this proves the lemma. 
\end{proof}
 
The center $\Delta \simeq \Z/2\Z$ of $W(\bF_4)$ is the commutator
subgroup of $A$ consisting of the transformations $\pm Id$.  
The main observation is the following: 
\begin{lemma}\label{f4lemma}
The $S_3\times S_3$-module $A/\Delta$ is irreducible. 
\end{lemma}
\begin{proof}
The only non-trivial invariant subspace of $A$ with respect to the
action of $U$ is the group $V_e$ consisting of transformations with 
even number of sign changes on the standard basis vectors. For
$u=(e_1+e_2+e_3+e_4)/2$ and $x=(x_1,x_2,x_3,x_4)\in \R^4$, 
the reflection $s_u$ based at
$u$ is given by, 
\[ s_u(x)=x-\left(\sum_{i=1}^4x_i\right)u.\]
Let $T_{12}$ be a transformation in $V_e$ sending $e_1$ and $e_2$ to
$-e_1$ and $-e_2$ respectively and fixing $e_3, e_4$. Now, 
\[
s_uT_{12}s_u(e_1)=s_uT_{12}(e_1-u)=s_u(-e_1-(-e_1-e_2+u))=s_u(e_2-u)=e_2.\]
Thus the group $V$ does not leave $V_e$ invariant, and this proves the
lemma. 
\end{proof}

\subsection{Proof of Theorem \ref{nonquotient}}
We now come to the proof of Theorem \ref{nonquotient} in the case when
$G_1$ is of type $\bF_4$ and $G_2$ is of type $\bB_4$. Let $L$ be the
compositum of Galois extensions  $L_1,\cdots, L_r$ over $K$, such that
$G(L_i/K), ~i=1,\cdots, r$ is isomorphic to a normal subgroup of
$W(\bB_4)$. Let $E_i\subset L_i$ be the Galois extension of $K$, such
that $G(L_i/E_i)$ is the maximal normal $2$-group in $G(L_i/E_i)$. 
The Galois group $G(E_i/K)$ is a normal subgroup of $S_3$. 
Let $E=E_1\cdots E_r$ denote the compositum of the fields $E_i$. 

The group $G(L_iE/E)$ is a normal $2$-subgroup of
$N_i=G(L_i/K)$, which is a normal 
subgroup  of $W(\bB_4)$. Denote by $\overline{N}_i$ the image of $N_i$ 
with respect to the projection $W(\bB_4)\simeq C\rtimes S_3\to S_3$,
where $C$ is a $2$-group. Let $C_i=C\cap N_i$. It follows from the
structure of $W(\bB_4)$, that there is a filtration of $C_i$ by
$\overline{N}_i$-stable subgroups, such that the graded components are
abelian $2$-groups of cardinality at most $4$.

For any $i$, consider the projection map $G(L/E)\to G(L_iE/E)$. By
Lemma \ref{compaction}, this map 
is equivariant as modules for the projection map $G(E/K)\to
G(E_i/K)$.  It follows that there is a filtration of 
that there is a filtration of $G(L/E)$ by
$G(E/K)$-stable subgroups, such that the graded components are
abelian $2$-groups of cardinality at most $4$. 
 
Suppose there is a surjection from $G(L/K)\to W(\bF_4)$. Since $G(L/E)$
is a normal $2$-group, its image will land inside the subgroup $A$ of 
$ W(\bF_4)$. This induces a surjection from $G(E/K)\to S_3\times
S_3$. By Lemma \ref{f4lemma}, the image of $G(L/E)$ is contained
inside the center $\Delta$ of $W(\bF_4)$. 

Going modulo $\Delta$, we get a surjection from $G(E/K)\to
W(\bF_4)/\Delta$. Since  $W(\bF_4)/\Delta$ has no normal $3$-groups, 
the normal $3$-group contained inside $G(E/K)$  maps trivially to
$W(\bF_4)/\Delta$.  The
quotient of $G(E/K)$ by this normal $3$-grou is a $2$-group, and there
cannot be a surjection to $W(\bF_4)/\Delta$. This proves Theorem
\ref{nonquotient} in this case. 

\begin{ack} We thank Dipendra Prasad for useful
  discussions. 
\end{ack}


\begin{thebibliography}{xx}

\bibitem[BPR]{BPR} C. Bhagawat, S. Pisolkar, C. S. Rajan, {\em Commensurability and representation equivalent arithmetic lattices}, Int. Math. Res. Not., {\bf no. 8} (2014) 2017--2036.  

\bibitem[Bou]{Bou} N. Bourbaki, {\em Groupes et alg\`{e}bres de Lie,
    Chap. IV, V, VI}, 2\`{e}me \'{e}dition, Masson, Paris, 1981.

\bibitem[PlR]{PlR} V. P. Platonov and A. S. Rapinchuk, {\em Algebraic
Groups and Number Theory}, Academic Press, New York, 1994. 

\bibitem[PR1]{PRr} G. Prasad and A. S. Rapinchuk,  {\em Existence of irreducible
$\R$-regular elements in Zariski dense subgroups}, Math. Res. Letters,
{\bf 10} (2003) 21--32. 

\bibitem[PR2]{PRwc} G. Prasad and A. S. Rapinchuk, {\em Weakly
commensurable arithmetic groups and isospectral locally symmetric
spaces}, { Publ. Math. Inst. Hautes tudes Sci.}  {\bf 109} (2009),
113--184.

\bibitem[PR3]{PRl} G. Prasad and A. S. Rapinchuk, {\em On the fields 
generated by the lengths of closed geodesics in locally symmetric spaces}, 
{Geom. Dedicata} {\bf 172} (2014), 79--120.  
 
\bibitem[PR4]{PRg} G. Prasad and A. S. Rapinchuk, {\em Generic elements in 
Zariski-dense subgroups and isospectral locally symmetric spaces. Thin groups
and superstrong approximation}, Math. Sci. Res. Inst. Publ.,
{\bf 61}, Cambridge Univ. Press, Cambridge, (2014) 211--252.


\bibitem[T]{T} J. Tits, {\em Classification of algebraic semisimple
groups}, {Algebraic Groups and Discontinuous Groups}, Proc. Sym. Pure
Math. {\bf 9} 33--62, Amer. Math. Soc., Providence, 1966.



\bibitem[V]{V} 
V. E. Voskresenskii, {\em 
Maximal tori without effect in semisimple algebraic groups},
Mat. Zametki, 44:3 (1988), 309-318; 
Mathematical notes of the Academy of Sciences of the USSR 1988, 
Volume 44, Issue 3, pp 651-655.


\end{thebibliography}
\end{document}